\documentclass[a4paper,11pt]{article}

\usepackage[utf8]{inputenc}
\usepackage[T1]{fontenc}
\usepackage{lmodern}
\usepackage{tabularx}
\usepackage{color}
\usepackage[left=2.5cm, right=2.5cm, top=3cm, bottom=3cm]{geometry}

\usepackage[pdftex]{graphicx}
\usepackage{latexsym}
\usepackage{amsmath,amsthm}
\usepackage{amsfonts}
\usepackage{amssymb}
\usepackage{mathtools}

\usepackage{enumitem}

\allowdisplaybreaks

\makeatletter
\newcounter{step}

\def\step#1{\stepcounter{step}\gdef\@currentlabel{\thestep}\par\medskip\noindent\textit{Step \thestep.}\hskip0.6em }
\makeatother

\newtheorem{theorem}{Theorem}[section]
\newtheorem{lemma}[theorem]{Lemma}

\theoremstyle{definition}

\newtheorem{annahme}[theorem]{Assumption}
\newtheorem{example}[theorem]{Example}
\theoremstyle{remark}
\newtheorem{remark}[theorem]{Remark}

\def\cell {C_\ell}
\def\Ell{\mathrm{L}}
\def\He{\mathrm{H}}
\def\We{\mathrm{W}}
\newcommand{\e}{\mathrm{e}}
\newcommand{\dx}[1][x]{\,\mathrm{d}#1}
\def\ee{\e}
\def\RR{\mathbb{R}}
\def\NN{\mathbb{N}}
\def\tr{\mathop{\mathrm{tr}}}

\def\OM{\varOmega}
\def\RHO{\rho}

\title{On exponential splitting methods for semilinear abstract\\ {Cauchy problems}}
\author{Bálint Farkas, Birgit Jacob and Merlin Schmitz\\[2ex]School of Mathematics and Natural Sciences, IMACM\\ University of Wuppertal, Gaußstraße 20, 42119 Wuppertal, Germany\\
 E-mail: \{farkas,bjacob,meschmitz\}@uni-wuppertal.de}
\date{}
\begin{document}

\maketitle

\begin{abstract}
	
Due to the seminal works of Hochbruck and Ostermann \cite{HO05},  \cite{HO10} exponential splittings are well established numerical methods utilizing  operator semigroup theory for the treatment of semilinear evolution equations whose principal linear part involves a sectorial operator with angle greater than $\frac\pi 2$ (meaning essentially the holomorphy of the underlying semigroup). The present paper contributes to this subject by relaxing on the sectoriality condition, but in turn requiring that the semigroup operators act consistently on an interpolation couple (or on a scale of Banach spaces). Our conditions (on the semigroup and on the semilinearity) are inspired by the approach of T.{} Kato \cite{KatoNS} to the local solvability of the Navier-Stokes equation, where the $\Ell^p$-$\Ell^r$ smoothing of the Stokes semigroup was fundamental. The present abstract operator theoretic result is applicable for this latter problem (as was already the result of Ostermann and Hochbruck), or more generally in the setting of \cite{HO05}, but also allows the consideration of examples, such as non-analytic Ornstein-Uhlenbeck semigroups or the Navier-Stokes flow around rotating bodies.
\end{abstract}
\textbf{Keywords:} Semilinear Cauchy problems, exponential splitting methods, convergence order, $C_{0}$-semigroups, scales of Banach spaces\\
\textbf{Classification:} 47N40, 65J08, 47D06, 65M15,  65M12
\section{Introduction}

We investigate the semilinear Cauchy problem
 \begin{align}\label{semilinear}
 \dot{u}(t) = Au(t)+ g(t,u(t)), \quad u(0) = u_{0}, \quad t \in [0,T],
 \end{align}
 where $[0,T]$ is a given time interval, $A\colon D(A)\subset X \to X$ generates a $C_{0}$-semigroup $(\e^{tA})_{t\geq 0}$ on a Banach space $X$, $u_{0} \in D(A)$ and $g\colon [0,T]\times V\rightarrow X$ is a locally Lipschitz continuous function, where $V$ is a Banach space as well. The precise conditions on $A$ and $g$ are described below.
We assume that the linear problem  $\dot u (t)=  Au (t)$ is well-posed, that is $A$ generates a $C_0$-semigroup, and we suppose that the solutions  can be calculated effectively with high precision (or are explicitly known). In this article we prove convergence estimates for the so-called \emph{exponential splitting methods}. This is a particular \emph{operator splitting method}, that is, a general procedure for finding (numerical) solutions of complicated evolution equations by reduction to subproblems, whose solutions  are then to be combined in order to recover the (approximate) solution of the compound problem.
The literature both on the functional and the numerical analysis sides are extremely extensive, see, e.g., the surveys \cite{Quispel}, \cite{HO-high}, \cite{HaO16}. The decomposition of the compound problem can be based on various things, such as: on physical grounds (say, separating advection and diffusion phenomena, e.g., \cite{HV}), by mathematical-structural reasons (separating linear and non-linear parts, see e.g., \cite{HO05,HO10}, \cite{HHO14}; separating the history and present in case of delay equations, see \cite{BCsF2}), etc. 
The starting point for exponential splitting methods is the definition of the mild solution of problem \eqref{semilinear}, that is the variation-of-constants formula:
\begin{align}\label{eqn:mild}
u(t) = \mathrm{e}^{tA}u(0)+\int_0^t \mathrm{e}^{(t-\tau)A} g(\tau,u(\tau))\dx[\tau].
\end{align}
First, we describe the method introduced by Hochbruck and Ostermann in \cite{HO05,HO10} in the case when $(\e^{tA})_{t\geq 0}$ is an analytic semigroup, see also \cite{HochbruckExpInt}.
An \emph{exponential integrator} is a time stepping method and it approximates the  convolution term on the right-side by a suitable quadrature rule in a given time step, where the effect of the linear propagator is not approximated but inserted precisely. Thus, for given time step $h>0$ and $t_n\coloneqq nh$ with $n\in\mathbb N$ and $nh\le T=Nh$, the solution $u(t_n)$ of the semilinear equation \eqref{semilinear}, given recursively by
 \begin{align}
 u(t_{n+1}) = \e^{hA}u(t_{n}) + \int_0^{h}\e^{(h-\tau)A}g(t_{n}+\tau,u(t_{n}+\tau))\dx[\tau],
 \end{align}
 is approximated by the $s$-stage Runge-Kutta approximation $u_n$, which is subject to the recursion
\[ u_{n+1} = \mathrm{e}^{hA}u_n+\int_0^h \mathrm{e}^{(h-\tau)A}\sum_{j=1}^s \ell_j(\tau)g(t_{n}+c_jh,U_{n,j}) \dx[\tau]\]
(the initial value $u_0=u(0)$ is known).
Here $s$ is a positive integer,  $c_1,\dots, c_s\in [0,1]$ are pairwise distinct and  $(U_{n,i})_{i=1,\dots,s}$ are defined as the solution of the integral equation
\begin{align}
U_{n,i}= \e^{c_{i}hA}u_{n} + \int_0^{c_{i}h}\e^{(c_{i}h-\tau)A}\sum_{j=1}^s \ell_{j}(\tau)g(t_{n}+c_{j}h, U_{n,j})\dx[\tau]. \label{internalstage}
\end{align}
The values $U_{n,i}$ provide an  approximation of $u(t_{n} + c_{i}h)$ for $n \in \{0,\dots,N-1\}$ at internal steps and
 $\ell_{1}, \dots, \ell_s $ are Lagrange interpolation polynomials with nodes $c_1h, \dots, c_sh\in [0,h]$, thus $\sum_{j=1}^s \ell_{j}(\tau)g(t_{n}+c_{j}h, U_{n,j})$ yields an approximation of $g(t_n+\tau,u(t_n+\tau))$ for $\tau\in [0,h]$. We require the following conditions on the semilinear Cauchy problem \eqref{semilinear}.

 \begin{annahme}[The linear setting]\label{cond:lin}
\begin{enumerate}
\item $A\colon D(A)\subset X \to X$ generates a $C_{0}$-semigroup $(\e^{tA})_{t\geq 0}$ on a Banach space $X$ and $u_{0} \in X$.
\item  $(X,V)$ is an interpolation couple, that is, also $V$ is a  Banach space, $V$ and $X$ are (continuously) embedded in a topological vector space $\mathcal{X}$. 
We assume moreover that for each  $t>0$ the linear operator  $\e^{tA}$ leaves $V\cap X$ invariant and extends to a linear operator $\e^{tA} \in L(V)\cap L(X,V)$  with $ M\coloneqq \max_{t\in [0,T]}\{  \| \e^{tA}\|_{L(X)},  \| \e^{tA}\|_{L(V)}\}<\infty$.
\item $W$ satisfies the same condition as $V$ above. (Interesting will be the case $W\in\{X,V\}$.)
\item There is a continuous, non-increasing function $\RHO_X\colon (0,\infty)\to [0,\infty)$ with  $\RHO_X\in \Ell^1(0,T)$ such that
\[
\|\e^{t A}\|_{L(X,V)}\leq \RHO_X(t)\quad\text{for every $t\in (0,T]$}.
\]
And similarly, there is a continuous, non-increasing function $\RHO_W\colon (0,\infty)\to [0,\infty)$ with  $\RHO_W\in \Ell^1(0,T)$ such that
\[
\|\e^{t A}\|_{L(W,V)}\leq \RHO_W(t)\quad\text{for every $t\in (0,T]$}.
\]
\end{enumerate}
\end{annahme}
This set of conditions, together with the ones about the non-linearity (see Assumption \ref{cond:sol} below), is inspired by T.{} Kato's iteration scheme in his operator theoretic approach to the Navier-Stokes equations, see \cite{KatoNS} and Example \ref{examp:Stokes} below. He used the $\Ell^p$-$\Ell^r$ smoothing of the linear Stokes semigroup to ``compensate the unboundedness'' of the nonlinearity, and thus could apply Banach's fixed point theorem, just as it is required by the exponential splitting in the internal steps.

\noindent In the setting of Assumption \ref{cond:lin} we clearly have $M\ge 1$, and if we set  
\[
 \OM_X(h)\coloneqq \int_0^h \RHO_X(\tau)\dx[\tau] \quad\text{and}\quad \OM_{W}(h)\coloneqq \int_0^h \RHO_W(\tau)\dx[\tau],
 \]
   then $\OM_X$, $\OM_W$ is a monotone increasing, continuous function from $[0,T)$ to $[0,\infty)$ with $\OM_{X}(0)=\OM_{W}(0)=0$. Thus $\OM_X$ and  $\OM_W$ are so-called $\mathcal K$-functions. Moreover, we abbreviate $\RHO\coloneqq \RHO_X$ and $\OM\coloneqq \OM_X$, and set
 \begin{align*}
  C_\OM&\coloneqq \sup\Bigl\{ h\sum_{k=1}^n \|\e^{kh A}\|_{L(X,V)} \:\Big|\: 0<nh\le T\Bigr\}\\
  &
  \leq \sup\Bigl\{ h\sum_{k=1}^n \RHO(kh)\:\Big|\: 0<nh\le T\Bigr\}\leq \|\RHO\|_{\Ell^1(0,T)}<\infty. 
  \end{align*}
 
 The following example is motivated by the framework of the paper \cite{HO05} by Hochbruck and Ostermann.
 \begin{example}[Bounded, analytic semigroups]\label{examp:bddanasgrp}
 Let $A$ generate a bounded, analytic $C_0$-semi\-group $(\e^{tA})_{t\geq0}$ on $X$ and suppose (without loss of generality) that $0\in \rho(A)$. For a fixed $\alpha \in [0,1)$ we set $V \coloneqq D((-A)^{\alpha})$, the domain of the fractional power of $-A$, and equip it with the norm $\|v\|_{V}\coloneqq \|(-A)^{\alpha}v\|_{X}$.  Since the  semigroup operators  commute with the powers of the  generator we have $\| \e^{hA}\|_{L(V)} = \| \e^{hA}\|_{L(X)}$. Moreover, there exists a constant $C_A>0$ such that for $h> 0$ we have
 \begin{equation}\label{eq:holsgrpfracpow}
 \|\e^{hA}\|_{L(X,V)} \le C_Ah^{-\alpha}.
 \end{equation}
  (We refer to \cite[Ch.~3]{Haase}, \cite[Ch.~4]{LunardiInt}, or \cite[Ch.~9]{ISEM2012} for details concerning fractional powers of sectorial operators.)
  Thus for this example $ \RHO(h)\le C_Ah^{-\alpha}$ and $\RHO \in \Ell^1(0,1)$. Further,  $\OM(h)\leq\frac{C_A}{1-\alpha}h^{1-\alpha}$ and
 $  C_\OM\le C_A\frac{T^{1-\alpha}}{1-\alpha} $.

  We remark that the fractional powers for negative generators  of not necessarily analytic $C_{0}$-semigroups can be also defined, see,  \cite{Komatsu} and e.g., \cite[Sec.~II.5.c]{EN}, but the validity of an estimate as in \eqref{eq:holsgrpfracpow} for some $\alpha\in (0,1)$ implies analyticity of the semigroup,  see  \cite[Thm.~12.2]{Komatsu}.
  \end{example}
 
 More examples, also for non-analytic semigroups, are provided in Section \ref{sec:Ex} below.
 
  \begin{annahme}[Properties of the solution]\label{cond:sol}
 \begin{enumerate}
 \item The semilinear Cauchy problem \eqref{semilinear} has a unique mild solution $u$, that is, $u\colon [0,T]\rightarrow X$ and $u\colon (0,T]\to V$ are  continuous and $u$ satisfies the integral equation \eqref{eqn:mild}.
 \item \label{streifen} Let $r>0$ 
  and $g\colon [0,T] \times V \to X$ be bounded on the strip
  \begin{align*}
  S_{r} \coloneqq \{ (t,v)\in (0,T]\times V \, | \, \|v-u(t)\|_{V} \leq r\}
  \end{align*}
  around the  solution $u$ and  Lipschitz-continuous on $S_r$ in the second variable, i.e., there exists a real number $L>0$ such that for all $t \in (0,T]$ and $(t,v), (t,w)\in S_r$:
  \begin{align}\label{eqn:lipschitz}
  \|g(t,v) - g(t,w)\|_X \leq L \|v-w\|_{V}.
  \end{align}
 \item   The composition $f\colon [0,T] \to X$, with $f(t)\coloneqq g(t,u(t))$ satisfies $f\in \We^{{s,1}}([0,T],W)$ for a given natural number $s\ge 1$. Note that $\We^{1,1}([0,T],W)$ equals the set of absolutely continuous functions.
 \end{enumerate}
  \end{annahme}

\begin{remark}
	If $g\colon [0,T]\times V\to X$ is uniformly Lipschitz and bounded in the second variable on bounded sets in $V$, i.e., for each $B\subseteq V$ bounded there is $L_B\geq 0$ such that
    for all $t \in [0,T]$ and $(v,w)\in B$ one has
     \begin{align*}
     \|g(t,v) - g(t,w)\|_X \leq L_B \|v-w\|_{V},
     \end{align*}
	and the solution $u\colon [0,T]\to V$ is bounded, then  Assumption \ref{cond:sol}.\ref{streifen} is satisfied.
\end{remark}

The main result of this paper reads as follows.

\begin{theorem}\label{thmsemilinear}
	Suppose Assumption  \ref{cond:sol} and let the initial value problem \eqref{semilinear} satisfy Assumption \ref{cond:lin}. Then there exists a constant $C>0$, that only depends on $T,s,\ell_{i},S_{r}, g$, the space $V$ and the semigroup $(\e^{tA})_{t\ge 0}$, such that for $h$ sufficiently small and $0 \leq t_n = nh \leq T$, the approximation $u_n$ is well-defined, that is equation \eqref{internalstage} has a unique solution $U_{n,1},\dots,U_{n,s}\in V$ satisfying $(t_n+c_jh,U_{n,j})\in S_r$ for $j=1,\ldots, s$, and its error satisfies
\begin{align}\label{fehlerabschaetzung}
 \|u_{n}-u(t_{n})\|_{V} \leq C \cdot h^{s-1} \OM_W(h)\| f^{(s)}\|_{\Ell^1([0,t_{n}],W)}
\end{align}

\end{theorem}
It worths formulating the previous error estimate for the two special cases $W\in\{X,V\}$: For $W=V$,  we  can choose $\OM_V(h)=Mh$ and \eqref{fehlerabschaetzung} takes the form
\begin{align}\label{fehlerabschaetzung2}
 \|u_{n}-u(t_{n})\|_{V} \leq C \cdot h^{s} \| f^{(s)}\|_{\Ell^1([0,t_{n}],V)},
\end{align}
whereas for $W=X$ one relaxes the condition on $f$ and arrives at 
\begin{align}\label{fehlerabschaetzung3}
 \|u_{n}-u(t_{n})\|_{V} \leq C \cdot h^{s-1} \OM(h)\| f^{(s)}\|_{\Ell^1([0,t_{n}],X)}.
\end{align}
Suppose that $V$ equals the domain of the fractional power $(-A)^\alpha$ of the negative generator $-A$, where $\alpha\in (0,1)$ and $A$ is assumed to be the generator of an analytic semigroup if $\alpha>0$, see Example \ref{examp:bddanasgrp} above. This is setting of the paper  \cite{HO05} by Hochbruck and Ostermann. In this case we can take $\RHO(h)=ch^{-\alpha}$ and hence $\OM(h)=ch^{1-\alpha}$.
So the error estimate from Theorem  \ref{thmsemilinear} takes the form
\begin{align}\label{fehlerabschaetzung4}
 \|u_{n}-u(t_{n})\|_{V} \leq C \cdot h^{s-\alpha} \| f^{(s)}\|_{\Ell^1([0,t_{n}],X)}.
\end{align}
The paper \cite{HO05}  states the estimate (see (22) therein)
\begin{align}\label{fehlerabschaetzung5}
 \|u_{n}-u(t_{n})\|_{V} \leq C \cdot h^{s} \sup_{s\in [0,T]}\| f^{(s)}\|_X.
\end{align}
(Note that if $\alpha=0$, i.e., $X=V$ the proof in \cite{HO05} works also for non-analytic semigroups and the two bounds in \eqref{fehlerabschaetzung3} and \eqref{fehlerabschaetzung5} coincide.)
Our abstract approach does not recover the result in  \cite{HO05} as a special case, but we can remark the main novelty here: We do not require $V$ being a subspace of $X$, this allows for a larger flexibility. Problems that fit into this setting, beside the case of analytic semigroups, include non-analytic Ornstein-Uhlenbeck semigroups perturbed by non-linear potentials, Navier-Stokes equations in 3D, incompressible 3D flows around rotation obstacles, wave equation with a nonlinear damping, see  Section \ref{sec:Ex}.

The structure of the paper is as follows.
The proof of  Theorem \ref{thmsemilinear}. takes up the next section. To make the paper as self-contained as possible, some auxiliary results concerning Lagrange interpolation and Gronwall's lemma, are recalled in the Appendix. 
Finally, in Section \ref{sec:Ex} various examples, mentioned above, are presented.

\section{Proof of Theorem \ref{thmsemilinear}}\label{sec:Runge}

This section is devoted to the proof of Theorem \ref{thmsemilinear}.
We remark that for $s=1$  many of the sums in the following proof are empty, so equal $0$, and that the Lagrange ``interpolation'' polynomial $\ell_1\equiv 1$.
Let the initial value problem \eqref{semilinear} satisfy Assumptions \ref{cond:lin} and  \ref{cond:sol} with constants $M$, $r$
 and $L$. Further, let $\cell >0$ be given by Lemma \ref{lagrange2}, i.e., for all $h>0$
 \[
 |\ell_{i}(\tau)| \leq C_{\ell} \quad \text{for all } i \in \{1,\dots,s\} \text{ and for all } \tau \in [0,h].
 \]

For $n\in \mathbb N$, $h>0$  we define $C_{f,W}(n,h)$  by
\begin{align*}
C_{f,W}(n,h)&\coloneqq \int_0^h \|\e^{t A}\|_{L(W,V)} \dx[t]\, \| f^{(s)}\|_{\Ell^1([0,t_n],W)}=\OM_W(h)\| f^{(s)}\|_{\Ell^1([0,t_n],W)}\\
&\leq \OM_W(h)\| f^{(s)}\|_{\Ell^1([0,T],W)}.
\end{align*}

For the proof of Theorem \ref{thmsemilinear} the following lemmas are needed. Recall that $\OM_W(h),\OM(h)\to 0$ for $h\to 0$, so $C_{f,W}(n,h)\to 0$ for $h\to 0$ uniformly in $n\leq T/h$.

\begin{lemma}\label{lmm:internalbound}
Let $C>0$, $ \mathfrak{h}>0$ with 
\begin{align*}
MC   \mathfrak{h}^{s-1}C_{f,W}(n,\mathfrak{h}) +  \OM( \mathfrak{h})( s\cell +1)\max_{(t,y) \in S_r}\|g(t,y)\|_X \le r\quad\text{and}\quad \OM(\mathfrak{h}) \cell s L  <1.
 \end{align*} 
Suppose that for fixed $n\in\mathbb N$, $u_n\in V$ and $h\in (0,\mathfrak{h})$ with $(n+1)h\le T$ we have
\begin{align}\label{fehlerabschaetzung_n}
\|u_n-u(t_n)\|_V \le C h^{s-1}   C_{f,W}(n,h). 
\end{align}
Then, the equation \eqref{internalstage} has a unique solution $U_{n,1},\ldots, U_{n,s}\in V$ satisfying $(t_n+c_jh,U_{n,j})\in S_r$ for $j=1,\ldots, s$.
\end{lemma}

\begin{proof}
 The main idea is to show existence of $U_{n,1},\ldots, U_{n,s}\in V$ by means of Banach's fixed point theorem. We equip $V^s$ with the maximum norm over the norms of its $s$ components.
For  $i=1,\ldots, s$ we define
\begin{equation*}
Y_h^i \coloneqq \{ v\in V \, | \, \|u(t_{n}+c_{i}h)-v\|_V \leq r\},
\end{equation*}
 and $Y_h\coloneqq Y_h^1 \times \cdots \times Y_h^s\subset V^s$.
Further,  let $\Phi_{h}\colon Y_h \to Y_h$ defined by
 \begin{align*}
 (\Phi_{h}(x_{1},x_{2},\dots ,x_{s}))_i &= \e^{c_{i}hA}u_{n}+ \int_0^{c_{i}h} \e^{(c_{i}h-\tau)A}\sum_{j=1}^s \ell_{j}(\tau)g(t_{n}+c_{j}h,x_{j}) \dx[\tau].
 \end{align*}
First, we show that $\Phi_h(x)\in Y_h$ for $x\in Y_h$. Indeed this follows from the calculation
 \begin{align*}
 \MoveEqLeft[1] \| \Phi_{h}(x)_{i}- u(t_{n}+c_{i}h)\|_V \\
 &= \left\| \e^{c_{i}hA}u_{n}+ \int_0^{c_{i}h} \e^{(c_{i}h-\tau)A}\sum_{j=1}^s \ell_{j}(\tau)g(t_{n}+c_{j}h,x_{j})\dx[\tau]\right. \\
 &\hspace*{2cm} -\left. \e^{c_{i}hA}u(t_{n})- \int_0^{c_{i}h} \e^{(c_{i}h-\tau)A}g(t_{n}+\tau,u(t_{n}+\tau))\dx[\tau] \right\|_V \\
 &\leq M\|u_{n}-u(t_{n})\|_V + \OM(h)( s\cell +1) \max_{(s,y) \in S_r}\|g(s,y)\|_X \\
 &\le MC h^{s-1}   C_{f,W}(n,h)  +\OM(h)( s\cell +1) \max_{(s,y) \in S_r}\|g(s,y)\|_X\\
&\le r.
 \end{align*}
Next, we show that $\Phi_h$ is a strict contraction. Let $x=(x_{1},\dots,x_{s}), \tilde{x}=(\tilde{x}_{1},\dots,\tilde{x}_{s})\in Y_h$. 
Since $g$ is Lipschitz continuous on $S_r$ we obtain
 \begin{align*}
 \| \Phi_{h}(x)- \Phi_{h}(\tilde{x})\|_{V^s} &= \max_{i \in \{1,\dots,s\}}\| (\Phi_{h}(x))_i-(\Phi_{h}(\tilde{x})_i\|_V\\
 &= \max_{i \in \{1,\dots,s\}}\left\| \int_0^{c_{i}h} \e^{(c_{i}h-\tau)A}\sum_{j=1}^s \ell_{j}(\tau) \bigl( g(t_{n}+c_{j}h, x_{j})- g (t_{n}+c_{j}h, \tilde{x}_{j})\bigr) \dx[\tau]\right\|_V \\
 &\leq \max_{i \in \{1,\dots,s\}} \OM(h) \cell s\max_{j \in \{1,\dots,s\}} \left\| g(t_{n}+c_{j}h,x_{j})-g(t_{n}+c_{j}h,\tilde{x}_{j}) \right\|_X \\
 &\leq \OM(h) \cell s L  \|x-\tilde{x}\|_{V^s}
 \end{align*}
Thus the statement follows by Banach's fixed point theorem and the definition of $Y_h$.
\end{proof}

 \begin{proof}[Proof of Theorem \ref{thmsemilinear}]
   Let $h_0>0$ such that $\OM(h_0) \cell s L  \le \frac{1}{2}$ (possible by $\OM(h)\to 0$ for $h\to 0$).

   Plainly, there exists a constant $C_F>0$ such that for $n\in\mathbb N$, $\xi\in [t_n,t_{n+1}]$ and $i=1,\ldots,s$ we have
 \begin{align}\label{eqn:frac}
 \left| \frac{(t_{n}+c_i h-\xi)^{s-1}}{(s-1)!}\right|\le C_F h^{s-1}.
 \end{align}
We define
\begin{align*}
C_{G,1} &\coloneqq 2sM^2 \cell L\max\{M,\OM(h_0)\},\\
C_{G,2} &\coloneqq 2\max\left\{2\cell^2  Ls^2C_F(\OM(h) + M  C_\OM),  M \cell sC_F \right\} ,\\
C&\coloneqq C_{G,2} \exp(C_{G,1}C_\Omega +C_{G,1}).
\end{align*}
Let $ \mathfrak{h}\in(0,h_0)$ with 
\begin{align*}
MC \mathfrak{h}^{s-1} C_{f,W}(n,\mathfrak{h})  +  ( s\cell +1) \max_{(s,y) \in S_r}\|g(s,y)\|_X \OM( \mathfrak{h})\le r.
\end{align*}

It suffices to show  the following statement by induction over $n\in \mathbb N$, $n\ge 0$:
 \begin{quote}
 For $h\in(0,\mathfrak{h})$ and $n\in\mathbb N$ with $(n+1)h\le T$ equation \eqref{internalstage} has a unique solution $U_{n,1},\ldots, U_{n,s}\in V$ satisfying $(t_{n}+c_jh,U_{n,j})\in S_r$ for $j=1,\dots,s$, and
\begin{align}\label{fehlerabschaetzung_kurz}
\|u_n-u(t_n)\|_V \le C h^{s-1}   C_{f,W}(n,h). 
\end{align} 
\end{quote}

If $n=0$, then $u_0=u(t_0)$. Thus the norm estimate of $ \|u_{0}-u(t_{0})\|_{V} $ is trivial and the unique existence of $U_{0,1},\ldots, U_{0,s}\in V$ satisfying $(t_{0}+c_jh,U_{0,j})\in S_r$ follows from Lemma \ref{lmm:internalbound}.

Next, we assume that the statement holds for $0,\ldots,n$, for some $n\in\mathbb N$, and we aim to show the statement for $n+1$ with $(n+1)h\leq T$.
For $k=0,\ldots,n$ let 
\[
e_{k}\coloneqq u_{k}-u(t_{k}),\qquad E_{k,i} \coloneqq U_{k,i} - u(t_{k}+c_{i}h)
\]
(recall $t_k=kh$). We divide the proof in several steps.

\step{}\label{taylor}
We show
\begin{align*}
\sum_{j=1}^s &\ell_j(\tau) g(t_n+c_jh, U_{n,j})
 - g(t_n+\tau,u(t_n+\tau))\\
&=\sum_{j=1}^s \ell_j(\tau) (g(t_n+c_jh,U_{n,j}) - f(t_n + c_jh))
 + \sum_{j=1}^s \ell_j(\tau)\int_{t_n+\tau}^{t_n + c_jh} \frac{(t_n+ c_jh - \xi)^{s-1}}{(s-1)!} f^{(s)}(\xi) \dx[\xi].
\end{align*}
We can write
 \begin{align}
 \notag\sum_{j=1}^s& \ell_j(\tau) g(t_n+c_jh, U_{n,j})-g(t_n+\tau,u(t_n+\tau))\\
 \label{eq:step1}& =\sum_{j=1}^s \ell_j(\tau) \Bigl(g(t_n+c_jh,U_{n,j}) - f(t_n + c_jh) + f(t_n + c_jh) - f(t_n + \tau) \Bigr)
 \end{align}
 as $ \sum_{j=1}^s \ell_j(\tau) =1$ (see Lemma \ref{lagrange}).
 Taylor expansion yields
 \begin{align*}
 f(t_n + c_jh) - f(t_n+\tau) &= \sum_{k=1}^{s-1} \frac{f^{(k)}(t_n+\tau)}{k!} (c_jh-\tau)^k \\
 &\hspace*{1cm} +\int_{t_n+\tau}^{t_n + c_jh} \frac{(t_n+ c_jh - \xi)^{s-1}}{(s-1)!} f^{(s)}(\xi) \dx[\xi].
 \end{align*}
 Recall the following property of Lagrange interpolation polynomials, see \eqref{null} in Lemma \ref{lagrange}: For $k\leq s-1$
 \begin{equation}\label{eq:lagrangeid}
 \sum_{j=1}^s \ell_j(\tau)(c_jh-\tau)^k=0.
 \end{equation}
 
Inserting this into the equation \eqref{eq:step1} above finishes Step \ref{taylor} as
\begin{align*}
 \sum_{j=1}^s& \ell_j(\tau) g(t_n+c_jh, U_{n,j}) -g(t_n+\tau,u(t_n+\tau))\\
 & =\sum_{j=1}^s \ell_j(\tau) (g(t_n+c_jh,U_{n,j}) - f(t_n + c_jh))\\
 &\quad + \sum_{j=1}^s \ell_j(\tau)\left( \sum_{k=1}^{s-1} \frac{f^{(k)}(t_n+\tau)}{k!} (c_jh-\tau)^k\right. +\left.\int_{t_n+\tau}^{t_n + c_jh} \frac{(t_n+ c_jh - \xi)^{s-1}}{(s-1)!} f^{(s)}(\xi) \dx[\xi]\right)\\
 &=\sum_{j=1}^s \ell_j(\tau) (g(t_n+c_jh,U_{n,j}) - f(t_n + c_jh))\\
 &\quad +\sum_{k=1}^{s-1} \frac{f^{(k)}(t_n+\tau)}{k!} \underset{\qquad =\,0 \text{ by \eqref{eq:lagrangeid}}}{\underbrace{\sum_{j=1}^s \ell_j(\tau)(c_jh-\tau)^k}} + \sum_{j=1}^s \ell_j(\tau)\int_{t_n+\tau}^{t_n + c_jh} \frac{(t_n+ c_jh - \xi)^{s-1}}{(s-1)!} f^{(s)}(\xi) \dx[\xi].
 \end{align*}
\step{}   Let $\tilde{\delta}_{k+1}$ and $\delta_{k+1}$, $k=0,\dots,n$, be given by 
 \begin{align}
 \tilde{\delta}_{k+1}&\coloneqq \int_0^{h} \e^{(h-\tau)A} \sum_{i=1}^s \ell_i(\tau) (g(t_k+c_ih,U_{k,i}) - f(t_k + c_ih))\dx[\tau],\label{tildedelta_n}\\
 \delta_{k+1} &\coloneqq \int_0^h \e^{(h-\tau)A} \sum_{i=1}^s \ell_i(\tau)\int_{t_k+\tau}^{t_k + c_ih} \frac{(t_k+ c_ih - \xi)^{s-1}}{(s-1)!} f^{(s)}(\xi) \dx[\xi] \dx[\tau].\label{delta_n}
\end{align}
Then Step \ref{taylor} implies
 \begin{align*}
 e_{n+1} &= u_{n+1}-u(t_{n+1}) \\
 &= \e^{hA}e_n + \int_0^h \e^{(h-\tau)A} \sum_{i=1}^s \ell_i(\tau) \left( g(t_n+c_ih,U_{n,i}) - f(t_n +\tau) \right) \dx[\tau] \\
 &= \e^{hA}e_n + \int_0^{h} \e^{(h-\tau)A} \left(\sum_{i=1}^s \ell_i(\tau) (g(t_n+c_ih,U_{n,i}) - f(t_n + c_ih))\right) \dx[\tau]\\
 &\hspace*{1.7cm} + \int_0^{h} \e^{(h-\tau)A} \left( \sum_{i=1}^s \ell_i(\tau)\int_{t_n+\tau}^{t_n + c_ih} \frac{(t_n+ c_ih - \xi)^{s-1}}{(s-1)!} f^{(s)}(\xi) \dx[\xi]\right) \dx[\tau]\\
 &=\e^{hA}e_n + \tilde{\delta}_{n+1} + \delta_{n+1}.
 \end{align*}
Solving the recursion yields:
 \begin{align*}
 e_{n+1} &= \e^{hA}e_{n} + \tilde{\delta}_{n+1} + \delta_{n+1}\\
 &= \e^{hA}\left( \e^{hA}e_{n-1}+\tilde{\delta}_{n} + \delta_{n}\right) + \tilde{\delta}_{n+1} + \delta_{n+1}\\
 &= \sum_{k=0}^{n} \e^{khA} \left(\tilde{\delta}_{n+1-k} + \delta_{n+1-k}\right),
 \end{align*}
 as $e_0=0$. Hence
 \begin{align}
 \left\| e_{n+1}\right\|_{V} \nonumber
 &\leq \sum_{k=0}^{n}(\|\e^{khA}\tilde{\delta}_{n+1-k}\|_{V} + \| \e^{khA}\delta_{n+1-k}\|_{V})\nonumber\\
 &=  \sum_{k=0}^{n}(\|\e^{(n-k)hA}\tilde{\delta}_{k+1}\|_V + \|\e^{(n-k)hA} \delta_{k+1}\|_V).\label{eqn:absch}
 \end{align}
We will estimate the norms on the right-hand in \eqref{eqn:absch} side separately.
\step{} \label{part1} We start by bounding the Taylor-remainders  for each fixed $i$.
For $k=0,\ldots,n$ and $i=1,\ldots,s$ we define 
\begin{align}\label{def:Delta}
\Delta_{k,i} \coloneqq \int_0^{c_ih} \e^{(c_ih-\tau)A}\sum_{j=1}^s \ell_j(\tau)\int_{t_k+\tau}^{t_k + c_jh} \frac{(t_k+ c_jh - \xi)^{s-1}}{(s-1)!} f^{(s)}(\xi) \dx[\xi]\dx[\tau]
\end{align}
and  estimate
 \begin{align}
 \| \Delta_{k,i}\|_{V}
 &\leq \int_0^{c_ih} \| \e^{(c_ih-\tau)A}\|_{L(W,V)} \sum_{j=1}^s \left| \ell_{j}(\tau)\right| \left\| \int_{t_{k}+\tau}^{t_{k}+c_{j}h} \frac{(t_{k}+c_{j}h-\xi)^{s-1}}{(s-1)!}f^{(s)}(\xi) \dx[\xi]\right\|_{W}\dx[\tau]\nonumber\\
 &\leq  \cell \left(\int_0^{h} \| \e^{\tau A}\|_{L(W,V)} \dx[\tau] \right)  \sum_{j=1}^s \int_{t_{k}}^{t_{k}+h}\left| \frac{(t_{k}+c_{j}h-\xi)^{s-1}}{(s-1)!}\right| \|f^{(s)}(\xi)\|_{W} \dx[\xi]\nonumber\\
 &\overset{\eqref{eqn:frac}}{\leq} \cell sC_F  h^{s-1} \left(\int_0^{h} \| \e^{\tau A}\|_{L(W,V)} \dx[\tau]\right) \| f^{(s)}\|_{\Ell^1([t_{k},t_{k+1}],W)}.\label{eqndelta}
 \end{align}

\step{} \label{part1b}
We now consider the norm of $\delta_{k+1}$ (see \eqref{delta_n}).
We estimate as in Step \ref{part1} 
 \begin{align*}
 \|\delta_{k+1}\|_V &\le \cell sC_F  h^{s-1} \left(\int_0^{h} \| \e^{\tau A}\|_{L(W,V)} \dx[\tau]\right) \| f^{(s)}\|_{\Ell^1([t_{k},t_{k+1}],W)},
\end{align*}
and obtain 
\begin{align}
  \sum_{k=0}^{n}\|\e^{(n-k)hA}\delta_{k+1}\|_V &\le \sum_{k=0}^{n} M \cell sC_F  h^{s-1} \left(\int_0^{h} \| \e^{\tau A}\|_{L(W,V)} \dx[\tau]\right) \| f^{(s)}\|_{\Ell^1([t_{k},t_{k+1}],W)}\nonumber \\
&\le  \frac{1}{2}C_{G,2}  h^{s-1} C_{f,W}(n+1,h).\label{eqn:delta}
 \end{align}
 \step{} \label{EstE} We prove
 \begin{align}\label{eqn:Eni}
 \sum_{i=1}^s \|E_{k,i}\|_{V} \leq 2sM\|e_{k}\|_{V} + 2\cell s^2C_F  h^{s-1} \left(\int_0^{h} \| \e^{\tau A}\|_{L(W,V)} \dx[\tau]\right) \| f^{(s)}\|_{\Ell^1([t_{k},t_{k+1}],W)}.
\end{align}
 Thanks to Step~\ref{taylor} we can calculate
 \begin{align*}
E_{k,i} = U_{k,i} - u(t_k+c_ih)\\
 &= \e^{c_ihA} u_k +\int_0^{c_ih} \e^{(c_ih-\tau)A} \sum_{j=1}^s \ell_j(\tau) g(t_k+c_jh, U_{k,j})\dx[\tau]\\
 &\qquad - \left( \e^{c_ihA} u(t_k) +\int_0^{c_ih} \e^{(c_ih-\tau)A}g(t_k+\tau,u(t_k+\tau))\dx[\tau]\right)\\
 &=\e^{c_ihA}e_k + \int_0^{c_ih} \e^{(c_ih-\tau)A} \sum_{j=1}^s \ell_j(\tau) (g(t_k+c_jh,U_{k,j}) - f(t_k + c_jh))\dx[\tau]+ \Delta_{k,i},
\end{align*}
where $\Delta_{k,i}$ is given by \eqref{def:Delta}.
Using the Lipschitz-continuity of $g$ on $S_r$, we obtain
 \begin{align*}
 \MoveEqLeft[3.5]\left\| E_{k,i}\right\|_V \leq M\|e_k\|_V + \int_0^{c_{i}h} \left\|\e^{(c_ih-\tau)A} \right\|_{L(X,V)} \sum_{j=1}^s \left|\ell_j(\tau)\right|\\
 &\qquad \cdot \underset{\leq \, L \cdot \|E_{k,j}\|_V\, \text{ by }\eqref{eqn:lipschitz}}{\underbrace{\| g(t_k+c_jh,U_{k,j}) - g(t_k+c_jh,u(t_k+c_jh))\|_X}}\dx[\tau] + \left\| \Delta_{k,i}\right\|_V\\
 &\leq M\|e_{k}\|_{V} + \OM(h_0)\cell L \sum_{j=1}^s \left\| E_{k,j}\right\|_{V} + \left\|\Delta_{k,i}\right\|_{V}.
 \end{align*}
 Thus, using $\OM(h_0) \cell s L  \le \frac{1}{2}$ we conclude 
 \begin{align*}
 \sum_{i=1}^s \left\| E_{k,i}\right\|_{V} &\leq sM\|e_{k}\|_{V} + \frac{1}{2} \sum_{j=1}^s \left\| E_{k,j}\right\|_{V}+ \sum_{i=1}^s \|\Delta_{k,i}\|_{V}.
 \end{align*}
This together with \eqref{eqndelta} implies \eqref{eqn:Eni}.
 \step{} \label{part1c}
For $k=0,\ldots,n$, we now investigate the norm of $\tilde\delta_{k+1}$ (see \eqref{tildedelta_n}).
Using \eqref{eqn:Eni}, we obtain
\begin{align*}
\| \tilde{\delta}_{k+1}&\|_V \le \OM(h)\cell  L\sum_{i=1}^s \|E_{k,i}\|_{V}\nonumber\\
& \leq\OM(h)\cell  L\left(
2sM\|e_{k}\|_{V} + 2\cell s^2C_F  h^{s-1} \left(\int_0^{h} \| \e^{\tau A}\|_{L(W,V)} \dx[\tau]\right) \| f^{(s)}\|_{\Ell^1([t_{k},t_{k+1}],W)}\right).
\end{align*}
and
\begin{align*}
\|\tilde\delta_{k+1}&\|_X\le Mh\cell  L\sum_{i=1}^s \|E_{n,i}\|_{V}\nonumber\\
&\le  Mh\cell  L  \left(
2sM\|e_{k}\|_{V} + 2\cell s^2C_F  h^{s-1} \left(\int_0^{h} \| \e^{\tau A}\|_{L(W,V)} \dx[\tau]\right) \| f^{(s)}\|_{\Ell^1([t_{k},t_{k+1}],W)}\right).\nonumber
 \end{align*}
Thus, using $\|\e^{(n-k)hA}\|_{L(X,V)} h \le C_\OM$ 
we have
\begin{align}
  \sum_{k=0}^{n}&\|\e^{(n-k)hA}\tilde\delta_{k+1}\|_V\nonumber\\
&\le\|\tilde\delta_{n+1}\|_V + \sum_{k=0}^{n-1}\|\e^{(n-k)hA}\|_{L(X,V)} \|\tilde\delta_{k+1}\|_X
\nonumber\\
& \leq\OM(h)\cell  L\left(
2sM\|e_{n}\|_{V} + 2\cell s^2C_F  h^{s-1} \left(\int_0^{h} \| \e^{\tau A}\|_{L(W,V)} \dx[\tau]\right) \| f^{(s)}\|_{\Ell^1([t_{n},t_{n+1}],W)}\right)\nonumber\\
& \quad + Mh\cell  L  
2sM\sum_{k=0}^{n-1}\|\e^{(n-k)hA}\|_{L(X,V)}\|e_{k}\|_{V}\nonumber \\
& \quad + 2Mh\cell  L\cell s^2C_F  h^{s-1}\|\e^{(n-k)hA}\|_{L(X,V)} \left(\int_0^{h} \| \e^{\tau A}\|_{L(W,V)} \dx[\tau]\right) \| f^{(s)}\|_{\Ell^1([0,t_{n}],W)}\nonumber\\
& \leq  C_{G,1}\sum_{k=0}^{n-1}\|\e^{(n-k)hA}\|_{L(X,V)}h\|e_{k}\|_{V} +C_{G,1}\|e_{n}\|_{V} 
 +\frac{1}{2}C_{G,2} h^{s-1}C_{f,W}(n+1,h). \label{eqn:tildedeltaX}
 \end{align}

\step{} In the final step we estimate $e_{n+1}$. 
 Using \eqref{eqn:absch}, \eqref{eqn:delta} and \eqref{eqn:tildedeltaX}
 we obtain 
 \begin{align*}
 \left\| e_{n+1}\right\|_{V}&\le   \sum_{k=0}^{n}(\|\e^{(n-k)hA}\tilde{\delta}_{k+1}\|_V + \|\e^{(n-k)hA} \delta_{k+1}\|_V)\\
&\le    C_{G,1}\sum_{k=0}^{n-1}\|\e^{(n-k)hA}\|_{L(X,V)}h\|e_{k}\|_{V} +C_{G,1}\|e_{n}\|_{V} 
 +C_{G,2} h^{s-1}C_{f,W}(n+1,h),
\end{align*}
and Gronwall's Lemma (see Theorem \ref{gronwall1}) implies 
 \begin{align*}
 \| e_{n+1}\|_V &\leq C_{G,2}(1+ C_{G,1})\prod_{k=0}^{n-1}(1+C_{G,1}\|\e^{(n-k)hA}\|_{L(X,V)}h) h^{s-1} C_{f,W}(n+1,h)\\
&\leq C_{G,2} \exp\Bigl(C_{G,1} \sum_{k=0}^{n-1}\|\e^{(n-k)hA}\|_{L(X,V)}h +C_{G,1}\Bigr) h^{s-1} C_{f,W}(n+1,h)\\
&\leq C_{G,2} \exp\Bigl(C_{G,1}C_\Omega +C_{G,1}\Bigr) h^{s-1} C_{f,W}(n+1,h)\\
&=C\OM_W(h)h^{s-1}\| f^{(s)}\|_{\Ell^1([0,t_{n+1}],W)}.
 \end{align*}
\item By Lemma \ref{lmm:internalbound} applied to $n+1$, for $h\in(0,\mathfrak{h})$ if $(n+1)h\leq T$ equation \eqref{internalstage} has a unique solution $U_{n+1,1},\ldots, U_{n+1,s}\in V$ satisfying $(t_{n+1}+c_jh,U_{n+1,j})\in S_r$. \qedhere
\end{proof}

\section{Examples}\label{sec:Ex}

We present here examples for the situation described in Assumption \ref{cond:lin} and also for some admissible non-linearities satisfying Assumption \ref{cond:sol}.

\begin{example}[Gaussian heat semigroup]\label{examp:Gaussian}
	Consider the Gaussian heat semigroup $(\ee^{tA})_{t\geq 0}$ on $\Ell^2(\RR^d)$, for $t>0$ given as
\[
\ee^{tA}f\coloneqq g_t*f,
\]
where $g_t(x)=(4\pi t)^{-d/2}\ee^{-\frac{|x|^2}{4t}}$, $x\in \RR^d$, is the Gaussian kernel. Then, actually, $(\ee^{tA})_{t\geq 0}$ yields a consistent family of analytic $C_0$-semigroups on the whole $\Ell^p(\RR^d)$-scale, $p\in [1,\infty)$.  A short calculation using the Young convolution inequality yields for $1<p\leq r<\infty$ and $f\in \Ell^p(\RR^d)$ that
\[
\|\ee^{tA}f\|_{r}\leq c_{p,r} t^{-\frac{d}{2}(\frac1p-\frac1r)}\|f\|_p
\]
with an absolute constant $c_{p,r}$ (whose optimal value can be determined, cf.~\cite{BecknerAM}). As usual, we shall refer to this phenomenon as $\Ell^p$-$\Ell^r$-smoothing. We conclude that the choices $X=\Ell^p(\RR^d)$, $V=\Ell^r(\RR^d)$ and
\[
\OM(h)=c_{p,r} h^{1-\alpha}
\]
with $\alpha=\frac{d}{2}(\frac1p-\frac1r)$ are admissible choices in Assumption \ref{cond:lin}, provided $\frac{d}{2}(\frac1p-\frac1r)<1$. For similar estimates in case of symmetric, Markov semigroups we refer, e.g., to \cite[Ch.{} 2]{DaviesBook}.
	\end{example}

\begin{example}[Stokes semigroup]\label{examp:Stokes}
Similarly to the foregoing examples, $\Ell^p$-$\Ell^r$ smoothing is valid for the Stokes semigroup on the divergence free space $\Ell^p_\sigma(\RR^d)^d$, so $X=\Ell^p_\sigma(\RR^d)^d$, $V=\Ell^r_\sigma(\RR^d)^d$ and the $\OM$ from Example \ref{examp:Gaussian} are admissible in Assumption \ref{cond:lin}, see \cite{KatoNS}. 
	\end{example}

	\begin{example}[Ornstein-Uhlenbeck semigroups] \label{examp:OU}
		Let $Q\in \RR^{d\times d}$ be a positive semidefinite matrix and $B\in \RR^{d\times d}$. Suppose that the positive semidefinite matrix
		\[
		Q_t\coloneqq\int_0^t \ee^{sB}Q\ee^{sB^*}\dx[s]
		\]
		is invertible  for some $t>0$ (for this, a sufficient but not necessary, assumption is that $Q$ itself is invertible). Then $Q_t$ is invertible for all $t>0$, see \cite[Ch.{} 1]{Zabczyk}. Consider the Kolmogorov kernel
		\[
		k_t(x)\coloneqq \frac{1}{(4\pi)^{\frac d2} \det (Q_t)^{\frac12}} \ee^{-\frac14\langle Q_t^{-1} x,x\rangle},
		\]
		and for $t>0$ the operator $S(t)$ defined by
	\[
	S(t)f(x)\coloneqq(k_t*f)(\e^{tB}x)=\frac{1}{(4\pi)^{\frac d2} \det (Q_t)^{\frac12}}\int_{\RR^d} \ee^{-\frac14\langle Q_t^{-1} y,y\rangle}f(\ee^{tB}x-y)\dx[y].
	\]
	Then by Young's convolution inequality, we see that $S(t)$ acts indeed on $\Ell^p(\RR^d)$ for each fixed $p\in [1,\infty)$, it is linear and bounded with	
	\[
	\|S(t)\|_{\mathcal{L}(\Ell^p)}\leq \ee^{-\frac{\tr(B)}{p}t}.
	\]
Setting $S(0)=I$, we obtain a $C_0$-semigroup on $\Ell^p(\RR^d)$, called the Ornstein-Uhlenbeck semigroup, see, e.g., \cite{LMP} for details. The Ornstein-Uhlenbeck semigroup is, in general, not analytic on $\Ell^p(\RR^d)$ (see \cite{MetafuneOU} and \cite{FMPS}).
Similarly to Example \ref{examp:Gaussian} for $r\geq p$ and $f\in \Ell^p(\RR^d)$ we have for $t>0$ that
\[
\|S(t)f\|_r\leq \ee^{-\frac{\tr(B)}{r}t}\|k_t*f\|_r\leq  \ee^{-\frac{\tr(B)}{r}t}\|k_t\|_q\|f\|_p,
\] 		
with $1+\frac1r=\frac1p+\frac1q$. We can calculate
\begin{align}\label{eq:Qtest}
\|k_t\|^q_q&=\frac{1}{(4\pi)^{\frac {qd}2} \det (Q_t)^{\frac q2}}\int_{\RR^d}\ee^{-\frac q4\langle Q_t^{-1} y,y\rangle}\dx[y]
=\frac{1}{q^{\frac d2}(4\pi)^{\frac {qd}2} \det (Q_t)^{\frac q2}}\int_{\RR^d}\ee^{-\frac 14\langle Q_t^{-1} y,y\rangle}\dx[y]\\
\notag&=\frac{(4\pi)^{\frac {d}2} \det (Q_t)^{\frac 12}}{q^{\frac d2}(4\pi)^{\frac {qd}2} \det (Q_t)^{\frac q2}}=c_q\det(Q_t)^{-\frac{q}{2}(1-\frac1q)}=c_q\det(Q_t)^{-\frac{1}{2}(\frac1p-\frac1r)}.
\end{align}		
If $Q$ is invertible, then we have $\|Q_t^{-\frac12}\|\leq C t^{-\frac12}$, see, e.g., \cite{LunardiInv} (but also below).
Since $\det(Q_t^{-1})\leq C \|Q_t^{-1}\|^{d}$, we obtain $\det(Q_t)\geq C't^{-d}$, which, when inserted into \eqref{eq:Qtest}, yields
\[
\|k_t\|_q\leq c_{p,r,Q}t^{-\frac{d}{2}(\frac1p-\frac1r)},
\]
for some constant $c_{p,r,Q}$ depending on $p,r,Q$. This result, for invertible $Q$ is essentially contained in \cite{HieberSawada} (or \cite{Hansel} in an even more general situation of evolution families, see also \cite{GeissertLunardi}). It follows that $X=\Ell^p(\RR^d)$, $V=\Ell^r(\RR^d)$ and
\[
\OM(h)= c_{p,r,Q}h^{1-\frac{d}{2}(\frac1p-\frac1r)},
\]
are admissible choices in Assumption \ref{cond:lin} provided $r\geq p$, $\frac{d}{2}(\frac1p-\frac1r)<1$.

\medskip
Now if $Q$ is not necessarily invertible, but for some/all $t>0$ the matrix $Q_t$ is non-singular, then 
there is a minimal integer $n>0$  such that
\[
[Q^{\frac12},BQ^{\frac12},B^2Q^{\frac12},\dots,B^{n-1}Q^{\frac12}]\quad\mbox{has rank $d$,}
\]
see, for example, \cite[Ch.{} 1]{Zabczyk} (if $Q$ is invertible, then $n=1$).
One can show that in this case $\|Q_t^{-\frac12}\|\leq C t^{\frac12-n}$ (for $t$ near $0$), see, \cite[Lemma 3.1]{Lunardi} (and, e.g., \cite{Seidman,FarkasLunardi}), hence 
\[
\|k_t\|_q\leq ct^{-\frac{d(2n-1)}{2}(\frac1p-\frac1r)},
\]
i.e., $X=\Ell^p(\RR^d)$, $V=\Ell^r(\RR^d)$ and
\[
\OM(h)= c_{p,r,Q}h^{1-\alpha}
\]
with $\alpha=\frac{d(2n-1)}{2}(\frac1p-\frac1r)$ are admissible choices in Assumption \ref{cond:lin} provided $r\geq p$  and $r$ is near to $p$. Similar results hold for (strongly elliptic) Ornstein-Uhlenbeck operators on $\Ell^p(\Omega)$, $\Omega$ an exterior domain with smooth boundary, see \cite{GeissertHeckHieberWood}.
		\end{example}

\begin{example}[Ornstein-Uhlenbeck semigroups on Sobolev space] \label{examp:OU2}
	Consider again the Ornstein-Uhlenbeck semigroup $S$ from the foregoing example, given by $S(t)f(x)\coloneqq (k_t*f)(\e^{tB}x)$ for $t\geq 0$, $f\in \Ell^p(\RR^d)$ and $x\in \RR^d$. We have $\partial_x S(t)f(x)=(k_t*\partial_xf)(\e^{tB}x)\ee^{tB}$, hence $S(t)$ leaves $\We^{1,r}(\RR^d)$ invariant, and is locally uniformly bounded thereon. On the other hand $\partial_x S(t)f(x)=(\partial_x k_t * f)(\e^{tB}x)\ee^{tB}$, and analogously to Example \ref{examp:OU} one can prove that for $t>0$, $1\leq p\leq r$ and $f\in \Ell^p(\RR^d)$
\[
\|\partial_x S(t) f\|_r\leq  ct^{-\frac{d(2n-1)}{2}(\frac1p-\frac1r)+\frac{1}2-n}\|f\|_p.
\]
So if $n=1$, i.e., in the elliptic case, we obtain that $V=\We^{1,r}(\RR^d)$  and
\[
\OM(h)= c_{p,r,Q}h^{1-\alpha}
\]
with $\alpha=\frac{d}{2}(\frac1p-\frac1r)+\frac12$
are admissible choices in Assumption \ref{cond:lin} provided $r\geq p$  and $r$ is near to $p$ (if $\frac1p-\frac1r<\frac1d$).
\end{example}

\begin{example}[Stokes operator with a drift]
	Examples \ref{examp:Stokes} and \ref{examp:OU} can be combined. The operator $A$, defined by $Au(x)=\Delta u(x)+ Mx\cdot \nabla u(x)-Mu(x)$ (with appropriate domain) generates a (in general,  non-analytic) $C_0$-semigroup on the divergence free spaces $\Ell^p_\sigma(\Omega)^d$, $1<p\le r<\infty$ subject to $\Ell^p$-$\Ell^r$-smoothing, with $\Omega=\RR^d$, see \cite{HieberSawada}, $\Omega$ a bounded or an exterior domain, see \cite{GeissertHeckHieber}.
Thus $X=\Ell^p(\RR^d)^d$, $V=\Ell^r(\RR^d)^d$ and the $\OM$ from Example \ref{examp:Gaussian} are admissible in Assumption \ref{cond:lin}.
	\end{example}

\begin{example}
	Consider the Stokes-semigroup $S$ generated by $Au(x)=\Delta u(x)+ Mx\cdot \nabla u(x)-Mu(x)$ (with appropriate domain) on $\Ell^p_\sigma(\Omega)^d$ ($1<p<\infty$), where   $\Omega=\RR^d$ or $\Omega$ is a bounded or an exterior domain, see \cite{GeissertHeckHieber}.
	We then have for $t>0$, $1\leq p\leq r$ and $\Ell^p_\sigma(\Omega)^d$
	\[
	\|\nabla S(t) f\|_r\leq  ct^{-\frac{d}{2}(\frac1p-\frac1r)-\frac{1}{2}}\|f\|_p,
	\]
	where $1<p\leq r<\infty$. If $\frac 1p=\frac1s+\frac1r$, $d<s$ and $\frac1p-\frac 1s<\frac2d$, then $V=\Ell^s(\Omega)^d\cap \We^{1,r}(\Omega)^d$ with \[
\OM(h)= c_{p,r,Q}h^{1-\alpha}
\]
and $\alpha=\max\{\frac{d}{2r}+\frac12,\frac{d}{2s}\}$ is an admissible choice in Assumption \ref{cond:lin} and the non-linearity $g(u)=u\cdot \nabla u$ also satisfies the required Lipschitz conditions (see Example \ref{examp:nablanonlin}).
	\end{example}

\begin{example}[Ornstein-Uhlenbeck semigroups on spaces with invariant measures]
	Similar results as in Example \ref{examp:OU} are valid for the Ornstein-Uhlenbeck semigroup $(S(t))_{t\geq0}$ on spaces $\Ell^p(\RR^d,\mu)$ with invariant measures  $\mu$ (cf, e.g., \cite{LunardiInv,FarkasLunardi}). Note however that $(S(t))_{t\geq0}$ is analytic in this case (provided $p>1$), see, e.g. \cite{MPP}, \cite{MPRS-dom}, \cite{CFMP} also for further details.
\end{example}

\begin{example}[Interpolation spaces vs.~growth function]
	A special case of a theorem of Lunardi, \cite[Thm.{} 2.5]{LunardiEmb} yields some information, when the ``growth function'' $\OM$ can be taken to be of the form $\OM(h)=ch^{1-\alpha}$ for some $\alpha\in(0,1)$. Let $(S(t))_{t\geq 0}$ be a $C_0$-semigroup on the Banach space $X$ with generator $A$. Let $V\subseteq X$ be a further Banach space, and suppose that for some constants $\beta\in (0,1)$,  $\omega\in\RR$, $c>0$ one has
	\begin{equation*}
	\|S(t)\|_{L(X,V)}\leq \frac{c\ee^{\omega
	t}}{t^{\beta}}\qquad\text{for $t>0$,}
	\end{equation*}
and that for each $x\in X$ the function $(0,\infty)\ni t\mapsto S(t)x\in
V$ is measurable. Then  for the real interpolation spaces
\begin{equation*}
(X,D(A))_{\theta,p}\hookrightarrow (X,V)_{\theta/\beta,p}\qquad\text{
for all $\theta\in (0,\beta)$ and $1\leq p\leq\infty$}
\end{equation*}
(with continuous embedding).
\end{example}

 \begin{example}\label{ex:nonlin}
  Let $\alpha > 1$, $p \in [\alpha, \infty)$ and $U \subseteq \mathbb R^{d}$ be open. Then the map
  \[F \colon \Ell^{p}(U) \to \Ell^{\frac{p}{\alpha}}(U), \quad F(u) = \lvert u \rvert^{\alpha-1}u,\]
  is Lipschitz continuous on bounded sets. Furthermore, $F$ is real continuously differentiable with derivative
  \[F'(u)v= \lvert u\rvert^{\alpha-1}v + (\alpha-1)\lvert u\rvert^{\alpha-3}u\,\mathrm{Re}(u\bar v) \qquad\text{for } u,v \in \Ell^{p}(U).\]
  For a proof see \cite[Cor. 9.3]{ISEM2013}.
\end{example}

\begin{example}\label{examp:nablanonlin}
	The function $g\colon \Ell^s(\RR^d)^d\cap \We^{1,r}(\RR^d)^d\to \Ell^p(\RR^d)^d$, $g(u)=u\cdot \nabla u$ is Lipschitz continuous  on bounded sets if  $\frac 1p=\frac 1r+\frac 1s$. Indeed, for $u,v\in \Ell^s(\RR^d)^d\cap \We^{1,r}(\RR^d)^d$ we can write by H\"older inequality that
	\begin{align*}
	\|u\cdot\nabla u-v\cdot \nabla v\|_{\Ell^p}&\leq \|u\cdot\nabla u-u\cdot \nabla v\|_{\Ell^p}+\|u\cdot\nabla v-v\cdot \nabla v\|_{\Ell^p}\\
	&\leq C(\|u\|_{\Ell^s}\|u-v\|_{\We^{1,r}}+\|u-v\|_s\|v\|_{\We^{1,r}}),
	\end{align*}
proving the asserted Lipschitz continuity.
\end{example}

\begin{example}[Second order systems]
  The result can be applied to second order problems via the following technique.\\
  Consider the second order Cauchy problem
\begin{align}\label{cp2}
  \ddot{w}(t) &= Aw(t) + g(t,w(t),\dot w(t)) \quad , t \in [0,T] \\
  w(0)&=w_0 \in X , \quad \dot w(0)= w_{1} \in X \nonumber
\end{align}
on a Banach space $X$, where $A \colon D(A)\subset X \to X $ generates a Cosine function $(\mathrm{Cos}(t))_{t \in \RR}$.\\
The associated Sine function $\mathrm{Sin} \colon \RR \to  L(X)$ is given by
\[\mathrm{Sin}(t) \coloneqq \int_0^t \mathrm{Cos}(s)\dx[s].\]
We assume that the system \eqref{cp2} has a classical solution $w \in C^{2}$. Sufficient conditions for the existence of solutions can be found in \cite{TW78}.\\
Choosing $u = \left(\begin{smallmatrix} w\\ \dot w \end{smallmatrix}\right)$ we can rewrite \eqref{cp2} as a first order problem
\begin{align*}
\dot u(t) = \mathcal A u(t) + \tilde g(t, u(t)), \quad u(0)=u_{0},
\end{align*}
where $\mathcal A \coloneqq\left( \begin{smallmatrix}0&I \\ A & 0 \end{smallmatrix}\right)$, $\tilde g (t,u(t)) =\left( \begin{smallmatrix}0 \\ g(t,w(t),\dot w(t)) \end{smallmatrix}\right)$ and $u_{0} =\left( \begin{smallmatrix} w_{0}\\ w_{1} \end{smallmatrix}\right)$.\\
As stated in \cite[Thm. 3.14.11]{ABHN11} there exists a Banach space $V$ such that $D(A)\hookrightarrow V \hookrightarrow X$ and such that the part $\mathcal A$ of $\left(\begin{smallmatrix}
0 & I\\ A & 0
\end{smallmatrix}\right)$ in $V \times X$ generates a $C_{0}$-semigroup given by
\[T(t) = \begin{pmatrix}
\mathrm{Cos}(t) & \mathrm{Sin}(t) \\ A\mathrm{Sin}(t) & \mathrm{Cos}(t)
         \end{pmatrix} \quad , t \in \RR .\]
       We will illustrate this procedure with a short example which is presented in \cite[Ch. 9]{ISEM2013}.\\
       Consider the nonlinear wave equation with Dirichlet boundary conditions on a bounded open set $\emptyset \neq U \subseteq \mathbb R^{3}$ described by the system
       \begin{align}
         \ddot w(t) &= \Delta_{D}w(t) - \alpha w(t)\lvert w(t)\rvert^{2}, \quad t \in [0,T],\label{ex:wave-eq}\\
         w(0) &= w_{0}, \quad \dot w(0)= w_{1},\nonumber
       \end{align}
       where $w_{0} \in D(\Delta_{D})\coloneqq \{u \in \He^{1}_{0}(U) \mid \exists f \in \Ell^{2}(U) \ \forall v \in \He^{1}_{0}(U): \int_{U} \nabla u \cdot \nabla v \dx[x] = \langle f \mid v \rangle_{\Ell^{2}}\}$, $w_{1} \in \He^{1}_{0}(U)$ and $\alpha \in \mathbb R$ are given.\\
       We can reformulate this as the semilinear system
       \begin{align*}
         \dot u(t)= Au(t) + F(u(t)), \quad t \in [0,T], \quad u(0)=u_{0}
       \end{align*}
       on the Hilbert space $X = \He^{1}_{0}(U)\times \Ell^{2}(U)$ endowed with the norm given by $\| (u_{1},u_{2})\|^{2} = \| \lvert \nabla u_{1}\rvert \|_{2}^{2} + \|u_{2}\|_{2}^{2}$. Choose $u_{0}=(w_{0},w_{1})$ and
       \begin{align*}
         A &= \begin{pmatrix}
               0 & I \\ \Delta_{D} & 0
              \end{pmatrix} \qquad \text{with }D(A)= D(\Delta_{D})\times \He^{1}_{0}(U),\\
         F(u)&= (0, -\alpha u_{1}\lvert u_{1}\rvert^{2}) \eqqcolon (0,F_{0}(u_{1})) \; \text{for } u = (u_{1},u_{2}) \in X.
       \end{align*}
      As mentioned in Example \ref{ex:nonlin}, $F_{0}\colon \Ell^{6}(U)\to \Ell^{2}(U)$ is real continuously differentiable and Lip{\-}schitz continuous on bounded sets. Since $U \subseteq \mathbb R^{3}$, Sobolev's embedding yields $H_{0}^{1}(U)\hookrightarrow \Ell^{6}(U)$ thus $F\colon X \to X$ has the same properties.
       Now Duhamel's formula or the variation of constants in combination with the semigroup given above yield that the mild solution of \eqref{ex:wave-eq} satisfies
       \[w(t) = \mathrm{Cos}(t) w_{0} + \mathrm{Sin}(t)w_{1} + \int_0^t \mathrm{Sin}(t-s)F_{0}(w(s))\dx[s], \quad t \in [0,T],\]
       on $H_{0}^{1}(U)$. Thus the main theorem is applicable with $s=1$.
\end{example}
\appendix
\section{Lagrange interpolation polynomials}\label{sec:lagrange}

In this section we summarize some useful facts on Lagrange interpolation polynomials.

For fixed $s\in \mathbb{N}$ and different $ c_{1},\dots, c_{s}\in [0,1]$ denote by
$\ell_{i}$, $i=1,\dots,s$ the Lagrange basis polynomials with nodes $ c_{1}, \dots, c_{s}$, i.e.
 \[
 \ell_{j}(\tau)=\prod_{j=1\atop m\neq j}^s\dfrac{\tau-c_{m}}{c_{j}-c_{m}}.
 \]
 (If $s=1$, then $\ell_1\equiv 1$.)
Thus we have $\ell_{i}(c_{j})=1$ if $i=j$, and $\ell_{i}(c_{j})=0$ if $i \neq j$.
\begin{lemma}\label{lagrange}
 The Lagrange basis polynomials are of degree $s-1$ and satisfy
 \begin{align}\label{eins}
 \sum_{i=1}^s \ell_i(\tau)=1 \quad \text{for all } \tau \in \RR,
 \end{align}
 and for every integer $k$ with $1 \leq k \leq s-1$
 \begin{align}\label{null}
 \sum_{i=1}^s\ell_i(\tau)c_i^k = \tau^k\quad\text{and}\quad \sum_{i=1}^s\ell_i(\tau)(c_i-\tau)^k = 0 \quad \text{for all } \tau \in \RR.
 \end{align}
 (Note that for $s=1$ the foregoing statement is vacuously true.)
 \end{lemma}
 \begin{proof}

 The polynomial $\sum_{i=1}^s \ell_i(\tau)$ is of degree at most $s-1$ and equal $1$ at $s$ points $c_1, \dots, c_s$. So equation \eqref{eins} follows by the identity theorem.

 \medskip\noindent
 We show equation \eqref{null} by induction over $k$. For $k=1$ we can write by \eqref{eins}
 \begin{align*}
 q(\tau)\coloneqq \sum_{i=1}^s \ell_{i}(\tau)(c_{i}-\tau) = \sum_{i=1}^s \ell_{i}(\tau)c_{i} -\tau\sum_{i=1}^s \ell_{i}(\tau)=\sum_{i=1}^s \ell_{i}(\tau)c_{i} -\tau.
 \end{align*}
 Since each $\ell_{i}$ is of degree $s-1\geq 1$, the first expression on the right-hand side is of degree at most $s-1$. So $q$ is polynomial of degree at most $s-1$ with the $s$ zeroes $c_{1}, \dots, c_{s}$. By the identity theorem $q\equiv 0$, i.e.,
 \[
 \sum_{i=1}^s \ell_{i}(\tau)c_{i} = \tau.
 \]
Let $k<s-1$ and suppose that
 \[
\sum_{i=1}^s \ell_{i}(\tau)(c_{i}-\tau)^{m}=0\quad\text{and}\quad\sum_{i=1}^s \ell_{i}(\tau)c_{i}^{m} = \tau^{m} \quad\text{if } 1\leq m\leq k.
 \]
 We need to show that these hold also for $m=k+1$. By the binomial theorem
 \begin{align*}
 p(\tau)&\coloneqq \sum_{i=1}^s \ell_{i}(\tau)(c_{i}-\tau)^{k+1}= \sum_{i=1}^s \ell_{i}(\tau)\sum_{j=0}^{k+1} \binom{k+1}{j}c_{i}^{k+1-j}(-\tau)^{j} \\
 & = \sum_{i=1}^s \ell_{i}(\tau)c_{i}^{k+1}+ \sum_{i=1}^s \ell_{i}(\tau)\sum_{j=1}^k \binom{k+1}{j} c_{i}^{k+1-j}(-\tau)^{j} + (-\tau)^{k+1}\cdot \sum_{i=1}^s \ell_{i}(\tau)
 \intertext{by \eqref{eins} we can write}
   & = \sum_{i=1}^s \ell_{i}(\tau)c_{i}^{k+1}+ \sum_{j=1}^k \binom{k+1}{j}(-\tau)^{j}\sum_{i=1}^s \ell_{i}(\tau) c_{i}^{k+1-j} + (-\tau)^{k+1}
   \intertext{by the induction hypothesis we conclude}
 &= \sum_{i=1}^s \ell_{i}(\tau)c_{i}^{k+1}  + \sum_{j=1}^k  \binom{k+1}{j}(-\tau)^{j}\tau^{k+1-j} +(-\tau)^{k+1}\\
 &= \sum_{i=1}^s \ell_{i}(\tau)c_{i}^{k+1}  +\tau^{k+1} \sum_{j=1}^k  \binom{k+1}{j}(-1)^{j}+ (-\tau)^{k+1}\\
 &= \sum_{i=1}^s \ell_{i}(\tau)c_{i}^{k+1} -\tau^{k+1}.
 \end{align*}
Since $k<s-1$, the polynomial $p(\tau)$ has degree at most $s-1$ but $s$ zeros $c_1,\dots, c_s$. So that the identity theorem again yields that $p(\tau)\equiv 0$. By induction  equality \eqref{null} holds for all $k \leq s-1$.
 \end{proof}

\begin{lemma}\label{lagrange2}
 Let $0 \leq c_{1} < \cdots < c_{s}\leq 1$. There is a constant $C_\ell$ such that for all $h>0$ and for the Lagrange basis polynomials $\ell_1,\dots, \ell_s$ with nodes $ c_{1}h, \dots, c_{s}h$ one has
 \[
 |\ell_{i}(\tau)| \leq C_{\ell} \quad \text{for all } i \in \{1,\dots,s\} \text{ and for all } \tau \in [0,h].
 \]
 \end{lemma}
 Note that the constant $C_\ell$ depends only on the nodes $c_1,\dots c_s$  but not on $h$.
\begin{proof}
The assertion follows directly from the definition, since $|\tau/h - c_{m}|\leq 1 $ for $\tau \in [0,h]$ and $\min\{|c_m-c_j|:m\neq j\}>0$.
\end{proof}
\section{Discrete Gronwall Inequality}
The discrete Gronwall inequality will be crucial for the proof of our main result. We refer to \cite{Clark1987} for a short proof, even in a more general case.

\begin{theorem}[Discrete Gronwall Inequality]\label{gronwall1}
For $N\in \NN$ let $a_0,\dots, a_N\geq 0$, $b_0,\dots, b_N\geq 0$ and $z_0,\dots, z_N\in \RR$ be given.
		  Suppose that for each $n\in \{1,\dots,N\}$
		  \[
		  z_n\leq a_n+\sum_{j=0}^{n-1}b_jz_j.
		  \]
		  Then for each $n\in \{1,\dots,N\}$
		  \[
		   z_n\leq \left(\max_{j=0,\dots, n}a_j\right)\prod_{j=0}^{n-1}(1+b_j).
		  \]
\end{theorem}
\begin{proof}
   Here we recall the proof of the discrete Gronwall inequality from \cite{Clark1987}:  For $0\leq j\leq k\leq N$ set
  \[
 B_{k,j}\coloneqq \prod_{i=j}^{k-1}(1+b_i)^{-1}
 \]
 and notice that
 \[ B_{k,j+1}-B_{k,j}=\prod_{i=j+1}^{k-1}(1+b_i)^{-1}\Bigl(1-\frac{1}{1+b_{j}}\Bigr)=
 B_{k,j}b_{j}
 \]
 and for $i\leq j\leq k$
\[
B_{k,j}B_{j,i}=B_{k,i}.
\]
 Let $m\in \{0,\dots,N\}$ such that $z_m B_{m,0}=\max\{z_jB_{j,0}:j=0,\dots,N\}$.
By the assumption we have
		  \begin{align*}
		   z_m B_{m,0}&\leq aB_{m,0}+B_{m,0}\sum_{j=0}^{m-1}b_jz_j=aB_{m,0}+\sum_{j=0}^{m-1}B_{m,j}B_{j,0}b_jz_j\\
		   &\leq aB_{m,0}+z_mB_{m,0}\sum_{j=0}^{m-1}B_{m,j}b_j\\
		   &= aB_{m,0}+z_mB_{m,0}\sum_{j=0}^{m-1}(B_{m,j+1}-B_{m,j})\\
		   &= aB_{m,0}+z_mB_{m,0}(B_{m,m}-B_{m,0})=aB_{m,0}+z_mB_{m,0}-z_mB_{m,0}^2.
		  \end{align*}
		  By rearranging we arrive at
		  \begin{align*}
		   z_m B_{m,0}&\leq a.
		  \end{align*}
Since for each $n\in \{0,\dots,N\}$ one has $z_n B_{n,0}\leq z_m B_{m,0}\leq a$, the assertion is proved.
\end{proof}

\section*{Acknowledgement}

The authors are indebted to Petra Csom\'os for the many fruitful discussions and suggestions.

 \bibliographystyle{abbrv}
 \bibliography{exprk,examples}

\end{document}